\documentclass{amsart}
\usepackage{setspace}
\usepackage{a4}
\usepackage{amsthm}
\usepackage{latexsym}
\usepackage{amsfonts}
\usepackage{graphicx}
\usepackage{textcomp}
\usepackage{cite}
\usepackage{enumerate}
\usepackage{amssymb}
\usepackage{hyperref}
\usepackage{amsmath}
\usepackage{tikz}
\usepackage[mathscr]{euscript}
\usepackage{mathtools}
\newtheorem{theorem}{Theorem}[section]

\newtheorem{corollary}[theorem] {Corollary}

\title{This is the title}
\usepackage{fancyhdr}

\begin{document}
\hrule\hrule\hrule\hrule\hrule
\vspace{0.3cm}	
\begin{center}
{\bf\large{{Noncommutative Cauchy Bound and Noncommutative Montel Bound for Roots of Polynomials}}}\\
\vspace{0.3cm}
\hrule\hrule\hrule\hrule\hrule
\vspace{0.3cm}
\textbf{K. Mahesh Krishna}\\
School of Mathematics and Natural Sciences\\
Chanakya University Global Campus\\
NH-648, Haraluru Village\\
Devanahalli Taluk, 	Bengaluru  North District\\
Karnataka State, 562 110, India\\
Email: kmaheshak@gmail.com\\

Date: \today
\end{center}

\hrule\hrule
\vspace{0.5cm}
\textbf{Abstract}: In 1829, Cauchy derived an upper bound for every root of a complex polynomial using the maximum of the absolute values of the coefficients. In 1931, Montel derived an upper bound using the sum of the absolute values of the coefficients. We derive noncommutative versions of the Cauchy and Montel bounds.


\hrule

\hrule
\section{Introduction}

In 1829, Cauchy derived the following very important upper bound for every root of a complex polynomial using the maximum of the absolute values of the coefficients \cite{CAUCHY}.
\begin{theorem} \cite{HIRSTMACEY, RAHMANSCHMEISSER, MARDEN} \label{CB} (\textbf{Cauchy Upper Bound}) Let 
	\begin{align*}
		p(z)=a_0+a_1z+\cdots+a_{n-1}z^{n-1}+z^n\in \mathbb{C}[z].
	\end{align*}
	If $\lambda \in \mathbb{C}$ satisfies $p(\lambda)=0$, then 
	\begin{align*}
		|\lambda|< 1+\max_{0\leq j \leq n-1}|a_j|.
	\end{align*}
\end{theorem}
\begin{corollary} \label{CBC}\cite{HIRSTMACEY, RAHMANSCHMEISSER, MARDEN}
(\textbf{Cauchy Upper Bound}) Let 
\begin{align*}
	p(z)=a_0+a_1z+\cdots+a_{n-1}z^{n-1}+a_nz^n\in \mathbb{C}[z], \quad a_n \neq 0.
\end{align*}
If $\lambda \in \mathbb{C}$ satisfies $p(\lambda)=0$, then 
\begin{align*}
	|\lambda|< 1+\max_{0\leq j \leq n-1}\left|\frac{a_j}{a_n}\right|=1+\frac{1}{|a_n|}\max_{0\leq j \leq n-1}|a_j|.
\end{align*}	
\end{corollary}
By a suitable modification, Corollary \ref{CBC} gives the following result. 
\begin{theorem}\label{CLB}
\cite{HIRSTMACEY, RAHMANSCHMEISSER, MARDEN}
(\textbf{Cauchy Lower Bound}) Let 
\begin{align*}
	p(z)=a_0+a_1z+\cdots+a_{n-1}z^{n-1}+a_nz^n\in \mathbb{C}[z], \quad a_0\neq 0, a_n \neq 0.
\end{align*}
If $\lambda \in \mathbb{C}$ satisfies $p(\lambda)=0$, then 
\begin{align*}
|\lambda|&>\frac{1}{1+\max_{1\leq j \leq n}\left|\frac{a_j}{a_0}\right|}=\frac{1}{1+\frac{1}{|a_0|}\max_{1\leq j \leq n}|a_j|}\\
&=\frac{|a_0|}{|a_0|+\max_{1\leq j \leq n}\left|a_j\right|}.
\end{align*}		
\end{theorem}
In 1931, Montel derived another bound  for every root of a complex polynomial using the sum of the absolute values of the coefficients \cite{MONTEL}.
\begin{theorem} \cite{HIRSTMACEY, RAHMANSCHMEISSER, MARDEN}  \label{MB}   (\textbf{Montel Upper Bound}) Let 
	\begin{align*}
		p(z)=a_0+a_1z+\cdots+a_{n-1}z^{n-1}+z^n\in \mathbb{C}[z].
	\end{align*}
	If $\lambda \in \mathbb{C}$ satisfies $p(\lambda)=0$, then 
	\begin{align*}
		|\lambda|\leq \max\left\{1, \sum_{j=0}^{n-1}|a_j|\right\}.
	\end{align*}
\end{theorem}
\begin{corollary} \cite{HIRSTMACEY, RAHMANSCHMEISSER, MARDEN}
(\textbf{Montel Upper Bound}) Let 
\begin{align*}
	p(z)=a_0+a_1z+\cdots+a_{n-1}z^{n-1}+a_nz^n\in \mathbb{C}[z], \quad a_n \neq 0.
\end{align*}
If $\lambda \in \mathbb{C}$ satisfies $p(\lambda)=0$, then 
\begin{align*}
	|\lambda|\leq \max\left\{1, \sum_{j=0}^{n-1}\left|\frac{a_j}{a_n}\right|\right\}=\max\left\{1,\frac{1}{|a_n|} \sum_{j=0}^{n-1}|a_j|\right\}.
\end{align*}	
\end{corollary}
\begin{theorem} \label{MLB}\cite{HIRSTMACEY, RAHMANSCHMEISSER, MARDEN}
	(\textbf{Montel Lower Bound}) Let 
	\begin{align*}
		p(z)=a_0+a_1z+\cdots+a_{n-1}z^{n-1}+a_nz^n\in \mathbb{C}[z], \quad a_0\neq 0, a_n \neq 0.
	\end{align*}
	If $\lambda \in \mathbb{C}$ satisfies $p(\lambda)=0$, then 
	\begin{align*}
			|\lambda|&\geq \frac{1}{\max\left\{1, \sum_{j=1}^{n}\left|\frac{a_j}{a_0}\right|\right\}}=\frac{1}{\max\left\{1,\frac{1}{|a_0|} \sum_{j=1}^{n}|a_j|\right\}}\\
			&=\frac{|a_0|}{\max\left\{|a_0|, \sum_{j=1}^{n}\left|a_j\right|\right\}}.
	\end{align*}
\end{theorem}
We naturally ask for versions of Theorems \ref{CB},  \ref{CLB}, \ref{MB} and \ref{MLB} for polynomials over Banach algebras (particularly, over C*-algebras).  In this article, we derive them.

\section{Noncommutative Cauchy and Noncommutative Montel Bounds}
Let $\mathcal{A}$ be a unital Banach algebra and $\mathcal{A}[z]$ be the set of all polynomials over $\mathcal{A}$. 
We derive a non-commutative version of Theorem \ref{CB} as follows.

\begin{theorem} \label{NCB}
(\textbf{Noncommutative Cauchy Upper Bound}) Let $\mathcal{A}$ be a unital Banach algebra. Let 
\begin{align*}
	p(z)=a_0+a_1z+\cdots+a_{n-1}z^{n-1}+z^n\in \mathcal{A}[z].
\end{align*}
Let $u \in \mathcal{A}$ be such that $p(u)=0$ and 
\begin{align}\label{E1}
	\|u^n\|=\|u\|^n.
\end{align}
 Then 
\begin{align*}
		\|u\|< 1+\max_{0\leq j \leq n-1}\|a_j\|.
\end{align*}
\end{theorem}
\begin{proof}
	Define $M\coloneqq \max_{0\leq j \leq n-1}\|a_j\|.$ Let $u \in \mathcal{A}$ satisfies $p(u)=0$ and $\|u^n\|=\|u\|^n$. We claim that $	\|u\|< 1+M$. Let us suppose that the claim fails. Then 
	\begin{align*}
		   	\|u\|\geq 1+M>1.
	\end{align*}
Since $p(u)=0$, we have 
\begin{align*}
	u^n=-(a_0+a_1u+\cdots+a_{n-1}u^{n-1}).
\end{align*}
By taking the norm and using Equation (\ref{E1}), we get 
\begin{align*}
\|u\|^n&=\|u^n\|=\|-(a_0+a_1u+\cdots+a_{n-1}u^{n-1})\|\\
&=\|a_0+a_1u+\cdots+a_{n-1}u^{n-1}\|\\
&\leq \|a_0\|+\|a_1u\|+\cdots+\|a_{n-1}u^{n-1}\|\\
&\leq \|a_0\|+\|a_1\|\|u\|+\cdots+\|a_{n-1}\|\|u^{n-1}\|\\
&\leq \|a_0\|+\|a_1\|\|u\|+\cdots+\|a_{n-1}\|\|u\|^{n-1}\\
&\leq M(1+\|u\|+\cdots+\|u\|^{n-1})\|\\
&=M\frac{\|u\|^n-1}{\|u\|-1}.
\end{align*}
Hence 
\begin{align*}
\|u\|^n\leq M\frac{\|u\|^n-1}{\|u\|-1}. 
\end{align*}
Rearranging, 
\begin{align*}
	\|u\|-1\leq M\frac{\|u\|^n-1}{\|u\|^n}=M \left(1-\frac{1}{\|u\|^n}\right)<M.
\end{align*}
Therefore 
\begin{align*}
	1+M\leq \|u\|<1+M
\end{align*}
which is a contraction. Hence claim holds which completes the proof. 
\end{proof}
\begin{corollary}\label{GNCB}
(\textbf{Noncommutative Cauchy Upper Bound}) Let $\mathcal{A}$ be a unital Banach algebra. Let 
\begin{align*}
	p(z)=a_0+a_1z+\cdots+a_{n-1}z^{n-1}+a_nz^n\in \mathcal{A}[z], \quad a_n \text{ is invertible}.
\end{align*}
Let $u \in \mathcal{A}$ be such that $p(u)=0$ and 
\begin{align*}
	\|u^n\|=\|u\|^n.
\end{align*}
Then 
\begin{align*}
	\|u\|< 1+\max_{0\leq j \leq n-1}\|a_n^{-1}a_j\|\leq 1+\|a^{-1}_n\|\max_{0\leq j \leq n-1}\|a_j\|.
\end{align*}	
\end{corollary}
Note that we used condition (\ref{E1}) in the proof of Theorem \ref{NCB}. We are unable to derive the result without this assumption. We now derive noncommutative version of Theorem \ref{CLB}.
\begin{theorem}
(\textbf{Noncommutative Cauchy Lower Bound}) Let $\mathcal{A}$ be a unital Banach algebra. Let 
\begin{align*}
	p(z)=a_0+a_1z+\cdots+a_{n-1}z^{n-1}+a_nz^n\in \mathcal{A}[z], \quad a_0\text{ and } a_n \text{ are invertible}.
\end{align*}
Let $u \in \mathcal{A}$ be such that $p(u)=0$, $u$ is invertible  and 
\begin{align*}
	\|(u^{-1})^n\|=\|u^{-1}\|^n.
\end{align*}
Then 
\begin{align*}
	\frac{1}{\|u^{-1}\|}>\frac{1}{1+\max_{1\leq j \leq n}\|a_0^{-1}a_j\|}\geq \frac{1}{1+\|a^{-1}_0\|\max_{1\leq j \leq n}\|a_j\|}.
\end{align*}	
\end{theorem}
\begin{proof}
	Since $u$ invertible, we have 
	\begin{align*}
		0=a_n+a_{n-1}u^{-1}+\cdots+a_1u^{-(n-1)}+a_0u^{-n}.
	\end{align*}
Define $v\coloneqq u^{-1}$. Then 
\begin{align*}
	0=a_n+a_{n-1}v+\cdots+a_1v^{n-1}+a_0v^n.
\end{align*}
Result follows by applying Corollary \ref{GNCB} to $v$.
\end{proof}
\begin{corollary}
	Let $\mathcal{A}$ be a unital C*-algebra. Let 
	\begin{align*}
		p(z)=a_0+a_1z+\cdots+a_{n-1}z^{n-1}+a_nz^n\in \mathcal{A}[z], \quad a_n \text{ is invertible}.
	\end{align*}
	Let $u \in \mathcal{A}$ be such that $p(u)=0$ and $u$ is normal or self-adjoint.
	Then 
	\begin{align*}
		\|u\|< 1+\max_{0\leq j \leq n-1}\|a_n^{-1}a_j\|.
	\end{align*}
\end{corollary}
\begin{proof}
	Since $u$ is a normal element in a C*-algebra, it is known that  $\|u^{2^k}\|=\|u\|^{2^k}$ for all  $k \in \mathbb{N}$ \cite{TAKESAKI}. We claim that $	\|u^n\|=\|u\|^n$ for all  $k \in \mathbb{N}.$ Let us suppose that this fails for a nonzero normal element $u$. Let $n \in \mathbb{N}$ be such that $	\|u^n\|<\|u\|^n$. Choose $m$ such that $n+m=2^k$ for some $k \in \mathbb{N}$. Then 
	\begin{align*}
		\|u^{2^k}\|=\|u^{n+m}\|=\|u^nu^m\|\leq \|u^n\|\|u^m\|<\|u\|^n\|u\|^m=\|u\|^{n+m}=\|u\|^{2^k},
	\end{align*}
which is a contradiction. 
\end{proof}
Following the proof of Theorem \ref{NCB} we easily get following results. 
\begin{theorem}
(\textbf{Noncommutative Cauchy Upper Bound}) Let $\mathcal{A}$ be a unital Banach algebra. Let 
\begin{align*}
	p(z)=a_0+za_1+\cdots+z^{n-1}a_{n-1}+z^na_n\in \mathcal{A}[z], \quad a_n \text{ is invertible}.
\end{align*}
Let $u \in \mathcal{A}$ be such that $p(u)=0$ and 
\begin{align*}
	\|u^n\|=\|u\|^n.
\end{align*}
Then 
\begin{align*}
	\|u\|< 1+\max_{0\leq j \leq n-1}\|a_ja_n^{-1}\|\leq 1+\|a^{-1}_n\|\max_{0\leq j \leq n-1}\|a_j\|.
\end{align*}		
\end{theorem}
\begin{theorem}
(\textbf{Noncommutative Cauchy Lower Bound}) Let $\mathcal{A}$ be a unital Banach algebra. Let 
\begin{align*}
	p(z)=a_0+za_1+\cdots+z^{n-1}a_{n-1}+z^na_n\in \mathcal{A}[z], \quad a_0 \text{ and } a_n \text{ are invertible}.
\end{align*}
Let $u \in \mathcal{A}$ be such that $p(u)=0$, $u$ is invertible and 
\begin{align*}
	\|(u^{-1})^n\|=\|u^{-1}\|^n.
\end{align*}
Then 
\begin{align*}
	\frac{1}{\|u^{-1}\|}>\frac{1}{	1+\max_{1\leq j \leq n}\|a_ja_0^{-1}\|}\geq \frac{1}{1+\|a^{-1}_0\|\max_{1\leq j \leq n}\|a_j\|}.
\end{align*}			
\end{theorem}
\begin{corollary}
	Let $\mathcal{A}$ be a unital C*-algebra. Let 
	\begin{align*}
		p(z)=a_0+a_1z+\cdots+z^{n-1}a_{n-1}+z^na_n\in \mathcal{A}[z], \quad a_n \text{ is invertible}.
	\end{align*}
	Let $u \in \mathcal{A}$ be such that $p(u)=0$ and $u$ is normal or self-adjoint.
	Then 
	\begin{align*}
		\|u\|< 1+\max_{0\leq j \leq n-1}\|a_ja_n^{-1}\|\leq 1+\|a_n^{-1}\|\max_{0\leq j \leq n-1}\|a_j\|.
	\end{align*}
\end{corollary}

We now derive noncommutative version of Theorem \ref{MB}.
\begin{theorem} \label{NMB1}
	(\textbf{Noncommutative Montel Upper Bound}) Let $\mathcal{A}$ be a unital Banach algebra. Let 
	\begin{align*}
		p(z)=a_0+a_1z+\cdots+a_{n-1}z^{n-1}+z^n\in \mathcal{A}[z].
	\end{align*}
	Let $u \in \mathcal{A}$ be such that $p(u)=0$ and 
	\begin{align}\label{E2}
		\|u^n\|=\|u\|^n.
	\end{align}
	Then 
	\begin{align*}
		\|u\|\leq \max\left\{1, \sum_{j=0}^{n-1}\|a_j\|\right\}.
	\end{align*}
\end{theorem}
\begin{proof}
	Let $u \in \mathcal{A}$ be such that $p(u)=0$ and $\|u^n\|=\|u\|^n$. If $\|u\|\leq 1$, then the conclusion holds. We therefore assume that $\|u\|>1$. 
	Since $p(u)=0$, we have 
\begin{align*}
	u^n=-(a_0+a_1u+\cdots+a_{n-1}u^{n-1}).
\end{align*}
By taking norm and using Equation (\ref{E2}), we get 
\begin{align*}
	\|u\|^n&=\|u^n\|=\|-(a_0+a_1u+\cdots+a_{n-1}u^{n-1})\|\\
	&=\|a_0+a_1u+\cdots+a_{n-1}u^{n-1}\|\\
	&\leq \|a_0\|+\|a_1u\|+\cdots+\|a_{n-1}u^{n-1}\|\\
	&\leq \|a_0\|+\|a_1\|\|u\|+\cdots+\|a_{n-1}\|\|u^{n-1}\|\\
	&\leq \|a_0\|+\|a_1\|\|u\|+\cdots+\|a_{n-1}\|\|u\|^{n-1}\\
		&\leq \|a_0\|\|u\|^{n-1}+\|a_1\|\|u\|^{n-1}+\cdots+\|a_{n-1}\|\|u\|^{n-1}\\
	&=\|u\|^{n-1}\sum_{j=0}^{n-1}\|a_j\|.
\end{align*}
Therefore 	
\begin{align*}
	\|u\|^n\leq \|u\|^{n-1}\sum_{j=0}^{n-1}\|a_j\|.
\end{align*}
Cancellation of $\|u\|^{n-1}$ gives the result.
\end{proof}
Note that we used Equation (\ref{E2}) in the proof of Theorem \ref{NMB1}. We are unable to derive without it. 
\begin{corollary}
Let $\mathcal{A}$ be a unital Banach algebra. Let 
\begin{align*}
	p(z)=a_0+a_1z+\cdots+a_{n-1}z^{n-1}+a_nz^n\in \mathcal{A}[z], \quad a_n \text{ is invertible}.
\end{align*}
Let $u \in \mathcal{A}$ be such that $p(u)=0$ and 
\begin{align*}
	\|u^n\|=\|u\|^n.
\end{align*}
Then 
\begin{align*}
	\|u\|< \max\left\{1, \sum_{j=0}^{n-1}\|a_n^{-1}a_j\|\right\}\leq \max\left\{1,\|a_n^{-1}\| \sum_{j=0}^{n-1}\|a_j\|\right\}.
\end{align*}	
\end{corollary}
We now derive another noncommutative Montel upper bound with different assumption. 
\begin{theorem}\label{NMB}
(\textbf{Noncommutative Montel Upper Bound}) Let $\mathcal{A}$ be a unital Banach algebra. Let 
\begin{align*}
	p(z)=a_0+a_1z+\cdots+a_{n-1}z^{n-1}+z^n\in \mathcal{A}[z].
\end{align*}
Let $u \in \mathcal{A}$ be such that $p(u)=0$ and $u$ is invertible. Then
\begin{align*}
	\frac{1}{\|u^{-1}\|}\leq \max\left\{1, \sum_{j=0}^{n-1}\|a_j\|\right\}.
\end{align*}
\end{theorem}
\begin{proof}
Let $u \in \mathcal{A}$ satisfies $p(u)=0$ and  is invertible. If 
\begin{align*}
		\frac{1}{\|u^{-1}\|}\leq 1, 
\end{align*}
then the conclusion holds. Let us therefore assume that 
\begin{align*}
	\frac{1}{\|u^{-1}\|}> 1.
\end{align*}
Now we need to show that 
\begin{align*}
	\frac{1}{\|u^{-1}\|}\leq \sum_{j=0}^{n-1}\|a_j\|.
\end{align*}
Since $p(u)=0$ and $u$ is invertible, we have 
\begin{align*}
	1=-(a_{n-1}u^{-1}+\cdots+a_1u^{-(n-1)}+a_0u^{-n}).
\end{align*}
By taking the norm, we get 
\begin{align*}
	1&=\|-(a_{n-1}u^{-1}+\cdots+a_1u^{-(n-1)}+a_0u^{-n})\|\\
	&=\|a_{n-1}u^{-1}+\cdots+a_1u^{-(n-1)}+a_0u^{-n}\|\\
	&\leq \|a_{n-1}u^{-1}\|+\cdots+\|a_1u^{-(n-1)}\|+\|a_0u^{-n}\|\\
	&\leq \|a_{n-1}\|\|u^{-1}\|+\cdots+\|a_1\|\|u^{-(n-1)}\|+\|a_0\|\|u^{-n}\|\\
	&\leq \|a_{n-1}\|\|u^{-1}\|+\cdots+\|a_1\|\|u^{-1}\|^{n-1}+\|a_0\|\|u^{-1}\|^n\\
	&\leq  \|a_{n-1}\|\|u^{-1}\|+\cdots+\|a_1\|\|u^{-1}\|+\|a_0\|\|u^{-1}\|\\
	&=\|u^{-1}\|\left(\sum_{j=0}^{n-1}\|a_j\|\right).
\end{align*}
Therefore 
\begin{align*}
	\frac{1}{\|u^{-1}\|}\leq \sum_{j=0}^{n-1}\|a_j\|.
\end{align*}
\end{proof}
Note that we used the invertibility of $u$ in the proof of Theorem \ref{NMB}. We are unable to derive the result without this assumption. 
\begin{corollary}
(\textbf{Noncommutative Montel Upper Bound}) Let $\mathcal{A}$ be a unital Banach algebra. Let 
\begin{align*}
	p(z)=a_0+a_1z+\cdots+a_{n-1}z^{n-1}+a_nz^n\in \mathcal{A}[z], \quad a_n \text{ is invertible}.
\end{align*}
Let $u \in \mathcal{A}$ be such that $p(u)=0$ and $u$ is invertible. Then
\begin{align*}
	\frac{1}{\|u^{-1}\|}\leq \max\left\{1, \sum_{j=0}^{n-1}\|a_n^{-1}a_j\|\right\}\leq \max\left\{1,\|a_n^{-1} \|\left(\sum_{j=0}^{n-1}\|a_j\|\right)\right\}.
\end{align*}	
\end{corollary}
By doing a similar procedure as in the proof of noncommutative Cauchy lower bounds, we get following result. 
\begin{theorem}
(\textbf{Noncommutative Montel Lower Bound}) Let $\mathcal{A}$ be a unital Banach algebra. Let 
\begin{align*}
	p(z)=a_0+a_1z+\cdots+a_{n-1}z^{n-1}+a_nz^n\in \mathcal{A}[z], \quad a_0 \text{ and } a_n \text{ are invertible}.
\end{align*}
Let $u \in \mathcal{A}$ be such that $p(u)=0$ and $u$ is invertible. Then
\begin{align*}
	\|u\|\geq \frac{1}{\max\left\{1, \sum_{j=1}^{n}\|a_0^{-1}a_j\|\right\}}\geq  \frac{1}{\max\left\{1,\|a_0^{-1} \|\left(\sum_{j=1}^{n}\|a_j\|\right)\right\}}.
\end{align*}	
\end{theorem}

 \section{Conclusions}
 \begin{enumerate}
 	\item In 1829, Cauchy derived an upper bound for the zeros of complex polynomials using the maximum of the absolute values of the coefficients \cite{CAUCHY}. 
 	\item In 1931, Montel derived  an upper bound for the zeros of complex polynomials using the sum of the absolute values of the coefficients \cite{MONTEL}.
 	\item In this article, we  derived noncommutative versions of Cauchy and Montel bounds. 
 \end{enumerate}

 \bibliographystyle{plain}
 \bibliography{reference.bib}

\end{document}